\documentclass[11pt]{article}
\usepackage[numbers,sort&compress]{natbib}
\usepackage{enumerate}
\usepackage{amscd}
\usepackage{array}
\usepackage{amsmath}
\usepackage{latexsym}
\usepackage{amsfonts}
\usepackage{amssymb}
\usepackage{amsthm}
\usepackage{verbatim}
\usepackage{mathrsfs}
\usepackage{enumerate}
\usepackage{hyperref}

\theoremstyle{plain}
\theoremstyle{definition}\newtheorem{theorem}{Theorem}[section]
\theoremstyle{plain}\newtheorem{lemma}[theorem]{Lemma}
\theoremstyle{plain}\newtheorem{coro}[theorem]{Corollary}
\theoremstyle{plain}
\theoremstyle{remark}\newtheorem{remark}{Remark}[section]
\theoremstyle{definition}
\theoremstyle{plain}
\usepackage{xcolor}

\newcommand{\norm}[1]{\left\|#1\right\|}

\newcommand{\B}{\Big}

\newcommand{\be}{\begin{equation}}
\newcommand{\ee}{\end{equation}}
 \newcommand{\ba}{\begin{aligned}}
 \newcommand{\ea}{\end{aligned}}

  \newcommand{\f}{\frac}
    
  \newcommand{\ben}{\begin{enumerate}}
   \newcommand{\een}{\end{enumerate}}

\newcommand{\Rmnum}[1]{\expandafter\@slowromancap\romannumeral #1@}

\allowdisplaybreaks

\numberwithin{equation}{section}
\begin{document}
\title{Anisotropic Hardy-Sobolev inequality  in mixed Lorentz spaces with applications to the axisymmetric Navier-Stokes equations }
 \author{ Yanqing Wang\footnote{College of Mathematics and Information Science,  Zhengzhou University of Light Industry, Zhengzhou, Henan  450002,  P. R. China Email: wangyanqing20056@gmail.com},   \;~ Yike Huang\footnote{ College of Mathematics and Information Science, Zhengzhou University of Light Industry, Zhengzhou, Henan  450002,  P. R. China Email: huang\_yike@outlook.com},~~Wei Wei\footnote{School of Mathematics and Center for Nonlinear Studies, Northwest University, Xi'an, Shaanxi 710127, P. R. China Email: ww5998198@126.com },
   \;~~and ~~Huan Yu \footnote{ School of Applied Science, Beijing Information Science and Technology University, Beijing, 100192, P. R. China Email:  yuhuandreamer@163.com}
   }
\date{}
\maketitle
\begin{abstract}
 In this paper, we establish several new anisotropic Hardy-Sobolev inequalities in mixed Lebesgue  spaces and mixed Lorentz spaces, which covers many known corresponding results. As an application, this type of inequalities allows us to generalize some
 regularity criteria of the 3D axisymmetric Navier-Stokes equations.

 \end{abstract}

\section{Introduction}

\subsection{  Hardy-Sobolev   inequality}
The classical  Hardy-Sobolev inequality in \cite{[BCD],[Tao],[Hardy],[FS]} reads
  \be\label{hardy}\ba
 &\B\|\f{f(x)}{|x| }\B\|_{L^{p }(\mathbb{R}^{n})}\leq C
 \|\nabla f \|_{L^{p }(\mathbb{R}^{n})}, 1< p<n,\\
 &\B\|\f{f(x)}{|x|^{s}}\B\|_{L^{2}(\mathbb{R}^{n})}\leq C
 \|\Lambda^{s}f\|_{L^{2}(\mathbb{R}^{n})},   0\leq s<\f{n}{2}.
\ea\ee
 A variant of Hardy-Sobolev inequality involving fractional derivatives in the   $L^{p}$
 framework is proved in \cite{[MWZ],[zz]}
\be\label{HSF}
\B\|\f{f(x)}{|x|^{s}}\B\|_{L^{p}(\mathbb{R}^{n})}\leq C
 \|\Lambda^{s}f\|_{L^{p}(\mathbb{R}^{n})},   0\leq s<\f{n}{p}.
 \ee
 The Hardy-Sobolev  type inequality   and its extension play  an important role in the study of nonlinear elliptic equations \cite{[BV]}, heat equations with a singular potential \cite{[GZ]}, Morawetz-type estimates of wave equations \cite{[Tao]} and Schr\"odinger equations \cite{[zz]}, and
the Navier-Stokes equations \cite{[Yu],[CFZ],[CFZ1]}.

  Two kinds  of natural  refinement of the Hardy-Sobolev inequality are to study it in more general spaces and to consider it in the anisotropic case.
It is well known that the Lorentz spaces $L^{p,\infty}(\mathbb{R}^{n})$ and anisotropic Lebesgue spaces $L^{\overrightarrow{p} }(\mathbb{R}^{n})$ are both generalizations of the Lebesgue spaces  $L^{p }(\mathbb{R}^{n})$, where $\overrightarrow{p}=(p_1,p_2,\cdots,p_n)$.
 Making use of  the H\"older inequality in Lorentz spaces in \cite{[Neil]}  and the famous fact that $|x|^{-s}\in L^{\f{n}{s},\infty}(\mathbb{R}^{n})$,   Hajaiej-Yu-Zhai \cite{[HYZ]} established the following Hardy-Sobolev inequality   in Lorentz spaces
\be\label{HSL}
 \B\|\f{f(x)}{|x|^{s}}\B\|_{L^{p,q}(\mathbb{R}^{n})}\leq C
 \|\Lambda^{s}f\|_{L^{p,q}(\mathbb{R}^{n})}, 1<p<\infty,0<s<\f{n}{p}, 1 \leq q\leq\infty.
\ee
The first objective of this paper is to extend the Hardy-Sobolev inequality \eqref{hardy}-\eqref{HSF}  from usual  Lebesgue spaces
to   anisotropic Lebesgue spaces. More precisely, we state our first result as follows.
\begin{theorem}\label{the1.1}
Assume that $1<p_{i},q_{i}<\infty$ and $0<s<\f{n}{p_{i}}$, $i=1,2,\cdots, n$. Then there exists a positive constant $C$ such that
\be\label{HSAL}\ba
\B\|\f{f(x)}{|x|^{s}}\B\|_{L^{\overrightarrow{p} ,\overrightarrow{q}}(\mathbb{R}^{n})}
\leq  C
 \|\Lambda^{s}f\|_{L^{\overrightarrow{p},\overrightarrow{q}} (\mathbb{R}^{n})}.
\ea\ee
\end{theorem}
\begin{remark}
This theorem covers the known Hardy-Sobolev inequalities \eqref{hardy}-\eqref{HSL}. When $p_{i}=q_{i}$, the inequality \eqref{HSAL} reduces to
\be
\B\|\f{f}{|x|^{s}}\B\|_{L^{\overrightarrow{p} }(\mathbb{R}^{n})}\leq C
 \|\Lambda^{s}f\|_{L^{\overrightarrow{p} }(\mathbb{R}^{n})},
\ee
which still extends  the classical results  \eqref{hardy} and \eqref{HSF}.
\end{remark}
 To
illustrate the proof of the Hardy-Sobolev inequality \eqref{HSAL},
we give some comments on the role of  Lorentz spaces in the study of Hardy-Littlewood-Sobolev inequality and  Hardy-Sobolev  inequality in usual Lebesgue spaces.
Indeed, making full use of the fact $|x|^{-s}\in L^{\f{n}{s},\infty}(\mathbb{R}^{n})$ and the Young inequality in Lorentz spaces in \cite{[Neil]}, one can conclude the  classical Hardy-Littlewood-Sobolev inequality that
 $$
\B\|f\ast \f{1}{|x|^{n-s}}\B\|_{L^{\f{qn}{n-sq} }(\mathbb{R}^{n})} \leq C\B\|f\ast \f{1}{|x|^{n-s}}\B\|_{L^{\f{qn}{n-sq},q}(\mathbb{R}^{n})}   \leq C  \|f\|_{L^{q  }(\mathbb{R}^{n})}, 1<q<\f{n}{s}, 0<s<n,
 $$
 which also means that
\be\label{HLS2}
 \B\|f \B\|_{L^{\f{qn}{n-sq},q }(\mathbb{R}^{n})} \leq C\|\Lambda^{s}f\|_{L^{q  }(\mathbb{R}^{n})}, 1<q<\f{n}{s}.
 \ee
On the other hand, replacing the Young inequality by the H\"older inequality in    Lorentz spaces and using the  Sobolev inequality \eqref{HLS2}, one may derive the Hardy-Sobolev type inequality \eqref{HSL}.
This strategy will also be applied to prove \eqref{HSAL}. To this end, we observe that the Hardy-Sobolev inequality \eqref{HSAL} can be enhanced to the inequality
$$
\B\|\f{f(x)}{\prod_{i=1}^{n}|x_{i}|^{\f{s}{n}}}\B\|_{L^{\overrightarrow{p} ,\overrightarrow{q}}(\mathbb{R}^{n})}
\leq  C
 \|\Lambda^{s}f\|_{L^{\overrightarrow{p},\overrightarrow{q}} (\mathbb{R}^{n})}.
$$
This enlightens us to consider the     Hardy-Sobolev type inequality in anisotropic cases. It is worth remarking that this kind of anisotropic inequalities is very helpful to investigate the  cylindrical (axial symmetry) solutions of  partial differential equations.
  To consider the nonlinear elliptic equations
arising in astrophysics,  Badiale-Tarantello \cite{[BT]} first introduced the following  anisotropic Hardy-Sobolev inequality
\be\label{BTHSI}
\B\|r_{k}^{ -\f{s(n-q)}{2q(n-s)}}  f \B\|_{L^{\f{2q(n-s)}{n-q}}(\mathbb{R}^{n})}
\leq C\|\nabla f\|_{L^{q}(\mathbb{R}^{n})}, s<k, 2 \leq k\leq n,
\ee
where $r_{k}=\sqrt{x_{1}^{2}+x_{2}^{2}+\cdots+x_{k}^{2}}$.

Chen-Fang-Zhang \cite{[CFZ]}   proved the following generalization of Hardy-Sobolev inequality  \eqref{BTHSI}
\be\label{cfz}
\B\|r^{ {-\f{ s}{q }}}_{  { k}} f \B\|_{L^{q}(\mathbb{R}^{n})}
\leq C\|  f\|^{ \f{n-s}{q}-\f{n}{p}+1}_{L^{p}(\mathbb{R}^{n})}\|\nabla f\|^{ \f{n}{p}-\f{n-s}{q}}_{L^{p}(\mathbb{R}^{n})}, s<k, 2 \leq k\leq n,
\ee
and
employed it to consider the regularity of the three-dimensional axial symmetry  Navier-Stokes system.
Subsequently, the anisotropic Hardy-Sobolev inequality below
\begin{equation}\label{yu}
\Big\|\frac{f}{|r_2|^{1-\alpha}|x_{3}|^{\alpha}}
\Big\|_{L^p(\mathbb{R}^3)}\leq C\|\Lambda^\alpha_{x_{3}}\Lambda_{x_{1},x_{2}}^{1-\alpha}f\|_{L^p(\mathbb{R}^3)}\leq C  \|\nabla  f\|_{L^p(\mathbb{R}^3)}, 0<\alpha<1,
\end{equation}
was  proved and applied to investigate the regularity of axial
symmetry Navier-Stokes equations  by      Yu in  \cite{[Yu]}.
 Partially motivated by this and the proof of Theorem \ref{the1.1},  the second target of this paper is to further extend the aforementioned  Hardy-Sobolev inequality as follows.
\begin{theorem}\label{the1.2}
Let $\sum_{j=1}^{i}k_{j}\leq n$ and $k_{j}\in \mathbb{Z}^{+},1\leq j\leq i$.
Take any  $k_{1}$ elements from the set  $A_{0}=\{x_{1},x_{2},\cdots,x_{n}\}$
and denote $A_{1}=\{x_{i_{1}},x_{i_{2}},\cdots,x_{i_{k_{1}}}\}$,
  $r_{1}=\sqrt{x_{i_{1}}^{2}+x_{i_{2}}^{2}+\cdots+x^{2}_{i_{k_{1}}}}$.
Then choose any $k_{2}$ elements from the complementary set  $A_{0}\setminus A_{1}$  and  write $A_{2}=\{x_{i_{k_{1}+1}},x_{i_{k_{1}+2}},\cdots,x_{i_{k_{1}+k_{2}}}\}$,
 $r_{2}=\sqrt{x_{i_{k_{1}+1}}^{2}+x_{i_{k_{1}+2}}^{2}+\cdots+x^{2}_{i_{k_{1}+k_{2}}}}$.
 Continue until we take any $k_{i}$ elements from the complementary set  $A_{0}\setminus \bigcup_{j=1}^{i-1} A_{j}$  and
 denote $A_{i}=\{x_{i_{\Sigma_{j=1}^{i-1}k_{j}+1}},x_{i_{\Sigma_{j=1}^{i-1}k_{j}+2}},\cdots,x_{i_{\Sigma_{j=1}^{i}k_{j}}}\}$, $r_{i}=\sqrt{x_{i_{\Sigma_{j=1}^{i-1}k_{j}+1}}^{2}
 +x_{i_{\Sigma_{j=1}^{i-1}k_{j}+2}}^{2}
 +\cdots+x^{2}_{i_{\Sigma_{j=1}^{i}k_{j}}}}$.
Suppose that $1<p<\infty$ and $0<\alpha_{j}<\f{k_{j}}{p}$, $ 1\leq j\leq i$. Then for all $f\in C_{0}^{\infty}(\mathbb{R}^{n})$, there holds
\be\label{anHS}
\B\|\f{f(x_{1},x_{2},\cdots,x_{n})}{\prod^{i}_{j=1}|r_{j}|^{\alpha_{j}}}
\B\|_{L^{p}(\mathbb{R}^{n})}
\leq C\B\|\Lambda_{x_{i_{1}},\cdots,x_{i_{\Sigma_{j=1}^{i}k_{j}}}}^{\sum^{i}_{j=1}\alpha_{j}}f
\B\|_{L^{p}(\mathbb{R}^{n})}.
\ee
\end{theorem}
\begin{remark}
The  Mihlin-H\"ormander multipliers theorem on $L^p(\mathbb{R}^{n})$   guarantees that
\be\B\|\Lambda_{x_{i_{1}},\cdots,x_{i_{\Sigma_{j=1}^{i}k_{j}}}}^{\sum^{i}_{j=1}\alpha_{j}}f
\B\|_{L^{p}(\mathbb{R}^{n})}
\leq C\|\Lambda^{\sum^{i}_{j=1}\alpha_{j}}f\|_{L^{p}(\mathbb{R}^{n})}.
\label{MH}\ee
Therefore, one can replace  $\B\|\Lambda_{x_{i_{1}},\cdots,x_{i_{\Sigma_{j=1}^{i}k_{j}}}}^{\sum^{i}_{j=1}\alpha_{j}}f
\B\|_{L^{p}(\mathbb{R}^{n})}$  by  $ \|\Lambda^{\alpha_{1}
+\cdots+\alpha_{i}}f\|_{L^{p}(\mathbb{R}^{n})}$ in inequality  \eqref{anHS}.
\end{remark}
\begin{remark}
Setting $i=j=1$  and choosing $k_{1}=k$ and $A_{1}=\{x_{1},x_{2},\cdots,x_{k}\}$ in this theorem, one may derive from \eqref{anHS} and  \eqref{MH} that
$$
\B\|r_{k}^{-\f{ s}{q }}f \B\|_{L^{q}(\mathbb{R}^{n})}
\leq C   \| \Lambda ^{\f{ s}{q }} f \|_{L^{q}(\mathbb{R}^{n})}  \leq C\|  f\|^{ \f{n-s}{q}-\f{n}{p}+1}_{L^{p}(\mathbb{R}^{n})}\|\nabla f\|^{ \f{n}{p}-\f{n-s}{q}}_{L^{p}(\mathbb{R}^{n})}, s<k, 1 \leq k\leq n,
$$
where the    Gagliardo-Nirenberg inequality was used.
It is worth remarking that
this inequality implies Badiale-Tarantello's inequality \eqref{BTHSI}  and Chen-Fang-Zhang's inequality \eqref{cfz}.
\end{remark}
\begin{remark}
Taking $A_{1}=\{x_{1},x_{2}\}, A_{2}=\{x_{3}\}, \alpha_{1}=1-\alpha$, $ \alpha_{2}=\alpha$ in this theorem, one     has
$$
\Big\|\frac{f}{|r_1|^{1-\alpha}|x_{3}|^{\alpha}}
\Big\|_{L^p(\mathbb{R}^3)} \leq C  \|\nabla  f\|_{L^p(\mathbb{R}^3)},~\max\{1-\frac{2}{p},0\}<\alpha<\frac{1}{p},~p>1,
$$
which is a variant of Yu's inequality \eqref{yu}.
\end{remark}
 By a slightly modified proof of Theorem \ref{the1.2}, we can further obtain the following result.
\begin{coro}\label{coro}
Suppose that $\mathbb{R}^{n}=\mathbb{R}^{ k_{1}}\times\mathbb{R}^{ k_{2}}\times\cdots\times\mathbb{R}^{ k_{i}}\times\mathbb{R}^{ n-\sum_{j=1}^{i}k_{j}}$ and
$n\geq\sum_{j=1}^{i}k_{j}$.
Let
$r_{1}=\sqrt{x^{2}_{1}+x^{2}_{2}+\cdots+x^{2}_{k_{1}}}$,
$r_{2}=\sqrt{x^{2}_{k_{1}+1}+x^{2}_{k_{1}+2}+\cdots+x^{2}_{k_{1}+k_{2}}}$,
$\cdots$, $r_{i}=\sqrt{x_{\Sigma_{j=1}^{i-1}k_{j}+1}^{2}
 +x_{\Sigma_{j=1}^{i-1}k_{j}+2}^{2}
 +\cdots+x^{2}_{\Sigma_{j=1}^{i}k_{j}}}$,  $0<p,q\leq\infty$,  $1<(\overrightarrow{p_{j}})_{\ell}<\infty$, $1\leq(\overrightarrow{q_{j}})_{\ell}\leq\infty$ and $0<\alpha_{j}<\f{k_{j}}{ (\overrightarrow{p_{j}})_{\ell}}$, $1\leq j\leq i$,    $1\leq\ell\leq k_{j}$. Then for all $f\in C_{0}^{\infty}(\mathbb{R}^{n})$, there holds
\be\ba\label{anHS1}
&\B\|\f{f(x_{1},x_{2},\cdots,x_{n})}{\prod^{i}_{j=1}|r_{j}|^{\alpha_{j}}}
\B\|_{L^{\overrightarrow{p_{1}},\overrightarrow{q_{1}}}
(\mathbb{R}^{  k_{1}})\cdots L^{\overrightarrow{p_{i}},\overrightarrow{q_{i}}}(\mathbb{R}^{ k_{i}})L^{p,q}(\mathbb{R}^{ n-\sum_{j=1}^{i}k_{j}})}\\
\leq& C\B\|\Lambda_{x_{1},\cdots,x_{\Sigma^{i}_{j=1}k_{j}}}^{\sum^{i}_{j=1}\alpha_{j}}f
\B\|_{L^{\overrightarrow{p_{1}},\overrightarrow{q_{1}}}
(\mathbb{R}^{  k_{1}})\cdots L^{\overrightarrow{p_{i}},\overrightarrow{q_{i}}}(\mathbb{R}^{ k_{i}})L^{p,q}(\mathbb{R}^{ n-\sum_{j=1}^{i}k_{j}})}.
\ea \ee
\end{coro}

Next, we present an application of    Hardy-Sobolev   inequalities  obtained above to the axisymmetric Navier-Stokes equations. In what follows, for the convenience of presentation, we set $r=r_{2}=\sqrt{x^{2}_{1}+x^{2}_{2}}$ in $\mathbb{R}^{3}$.
\subsection{An application of Hardy-Sobolev   inequality   to the Navier-Stokes equations}
The three-dimensional axisymmetric Navier-Stokes equations   can be written in the cylindrical coordinates
\be\left\{\ba\label{ANS}
&\partial_{t}u_{r}+(u_{r}\partial_{r} +u_{x_{3}}\partial_{x_{3}} )u_{r}-(\partial_r^2+\partial_z^2+\frac{1}{r}\partial_r)u_r+\frac{1}{r^2}u_r-\frac{1}{r}(u_\theta)^2+\partial_rp=0, \\
&\partial_{t}u_{\theta}+(u_{r}\partial_{r} +u_{x_{3}}\partial_{x_{3}} )u_{\theta}-(\partial_r^2+\partial_z^2+\frac{1}{r}\partial_r)u_\theta+\frac{1}{r^2}u_\theta
+\frac{1}{r}u_\theta u_r=0, \\
&\partial_{t}u_{x_{3}}+(u_{r}\partial_{r} +u_{x_{3}}\partial_{x_{3}} )u_{x_{3}}-(\partial_r^2+\partial_{x_{3}}^2+\frac{1}{r}\partial_r)u_{x_{3}}+\partial_{x_{3}}p=0, \\
&\partial_r(ru_r)+\partial_{x_{3}}(ru_{x_{3}})=0,\\
&u_r|_{t=0}=u_{r}^0,~u_\theta|_{t=0}=u_{\theta}^0,~u_{x_{3}}|_{t=0}=u_{x_{3}}^0,
\ea\right.\ee
where, $u_r, u_\theta$ and $u_{x_{3}}$   denote the     radial  component, swirl component, vertical component of the velocity $u$,  respectively. It is valid that $u=u_r e_{r}+u_\theta e_{\theta}+u_{x_{3}} e_{x_{3}}$, where $e_{r}=(\f{x_{1}}{r},\f{x_{2}}{r},0)$, $e_{\theta}=(-\f{x_{2}}{r},\f{x_{1}}{r},0)$ and
$e_{x_{3}}=(0,0,1)$.
Let $\omega_r=-\partial_{x_{3}} u_\theta,$ $\omega_\theta=\partial_{x_{3}} u_r-\partial_r u_{x_{3}},$ $\omega_{x_{3}}=\frac{\partial_r (ru_\theta)}{r},$ then, there holds
$\omega=$curl$ u$$=\omega_re_r+\omega_\theta e_\theta+\omega_{x_{3}}e_{x_{3}}$.
The equations of $(\omega_r,\omega_\theta , \omega_{x_{3}})$ is determined by
\begin{eqnarray*}
	&&\partial_{t}\omega_r+(u_{r}\partial_{r} +u_{x_{3}}\partial_{x_{3}} )\omega_r -(\partial_r^2+\partial_{x_{3}}^2+\frac{1}{r}\partial_r)\omega_r+\frac{1}{r^2}\omega_r-(\omega_r\partial_r+\omega_{x_{3}}\partial_{x_{3}})u_r=0,\\[3mm]
	&& \partial_{t}\omega_\theta+(u_{r}\partial_{r} +u_{x_{3}}\partial_{x_{3}} )\omega_\theta-(\partial_r^2+\partial_{x_{3}}^2+\frac{1}{r}\partial_r)\omega_\theta+\frac{1}{r^2}\omega_\theta-\frac 2r u_\theta\partial_{x_{3}}u_\theta
	-\frac{1}{r}\omega_\theta u_r=0,\\[3mm]
	&&\partial_{t}\omega_{x_{3}}+(u_{r}\partial_{r} +u_{x_{3}}\partial_{x_{3}} )\omega_{x_{3}}
-(\partial_r^2+\partial_{x_{3}}^2+\frac{1}{r}\partial_r)\omega_{x_{3}}-(\omega_r\partial_r+\omega_{x_{3}}\partial_{x_{3}})u_{x_{3}}=0.
\end{eqnarray*}
For  the 3D axisymmetric Navier-Stokes equations \eqref{ANS} without
swirl ($u_{\theta}=0$),
it is well-known that  the global well-posedness was established independently by  Ukhovskii and Yudovich \cite{[UY]} and  Ladyzhenskaya
\cite{[L]}. See also \cite{[LMNP]} for a refined proof.
However, in the presence of swirl, the global well-posedness problem is still open.  By Caffarelli, Kohn, Nirenberg \cite{[CKN]}, the problem is reduced to how to remove the possible singularities on the symmetry axis.  Recently, using DeGeogi-Nash-Moser iterations and a blow-up approach respectively,  Chen-Strain-Tsai-Yau \cite{CSTY1,CSTY2} and Koch-Nadirashvili-Seregin-\v{S}ver\'ak  \cite{KNSS} obtained an interesting and important development on this problem. They proved that if the solution satisfies $(1) \ |ru(x,t)|\le C$ or $(2)\ |u(x,t)|\le \frac{C}{\sqrt{T^*-t}}$ for $0<t<T^*$, where $C>0$ is an arbitrary and absolute constant and $(0,T^*)$ is the maximal existence interval of the solution, then  the solution is globally regular. It should be remarked that these conditions are scaling invariant and imply the possible blow-up rate of the solution.
See also  various extensions by Pan \cite{[P]} and Lei and Zhang \cite{[LZ]}.

Since $ru_\theta$ is a scaling invariant
quantity, it is natural and  important to
consider the critical regularity conditions for $u_\theta$ for the axisymmetric Navier-Stokes
equations. The regularity criteria for $u_\theta$  are derived in \cite{[NP],[KP],[KPZ]}
A scaling invariant (Serrin type) regularity condition  based on $u_{\theta}$ is due to Zhang-Zhang \cite{[ZZ]}, where it is shown that the condition
$u_{\theta}\in L^{4}(0,T;L^{6}(\mathbb{R}^{3}))$
guarantees the regularity of solutions of the Navier-Stokes equations \eqref{ANS}.
Recently, Chen,   Fang and   Zhang \cite{[CFZ]}  showed that the solutions of   the Navier-Stokes equations \eqref{ANS} is regular on $(0,T]$ provided
\be\label{CFZ}
r^{d}u_{\theta}\in L^{q}(0,T; L^{p}(\mathbb{R}^{3})), \f2q+\f3p=1-d, 0\leq d<1, \f{3}{1-d}<p\leq\infty, \f{2}{1-d}\leq q\leq\infty,
\ee
or $ r^{d}u_{\theta}\in L^{\infty}(0,T; L^{\f{3}{1-d}}(\mathbb{R}^{3}))$ and its norm sufficiently small. The above regularity criteria of $u_{\theta}$, which is scaling invariant, develop the corresponding results \cite{[NP],[KP],[KPZ],[ZZ]}. The  key points are
 the general anisotropic Hardy-Sobolev inequality and combinations of the equations satisfied by $\Gamma=\frac{\omega_\theta}{r}$ and $\Phi=\frac{\omega_r}{r}$, i.e.
\be\left\{\ba\label{w}
&\partial_t \Gamma+u \cdot\nabla \Gamma-(\Delta+\frac2r \partial_r)\Gamma+2\frac{u_\theta}{r}\Phi=0, \\
&\partial_t \Phi+u \cdot\nabla \Phi-(\Delta+\frac2r \partial_r)\Phi-(\omega_r\partial_r+\omega_{x_{3}}\partial_{x_{3}})\frac{u_r}{r}=0.
\ea\right.\ee
 Very recently,  Chen,   Fang and   Zhang \cite{[CFZ1]}  proved that the lifespan  condition \eqref{CFZ} can be replaced by
\be\label{cfz2}
r^{\alpha}u^{\theta}\in L^{q}(0,T; L_{x_{3}}^{p_{3}}(\mathbb{R}  )L_{x_{2},x_{1}}^{p_{2},\infty}(\mathbb{R}^{2}  ) ), \f2q+\f1p_{3}+\f2p_{2} =1-\alpha, 0\leq \alpha<\f12,
\ee
Let  us also mention that Yu \cite{[Yu]} recently showed that
$|x_{3}|u_{\theta}\in L^{\infty}(0,T;  L^{\infty}(\mathbb{R}^{3})) $
and the sufficiently small norm of  $\||x_{3}|u_{\theta}\|_{L^{\infty} L^{\infty}}$ ensure the regularity of solutions of axisymmetric Navier-Stokes equations
with the help of    Hardy-Sobolev  inequality \eqref{yu}.
The other regularity criteria and recent studies  can be seen in \cite{[LZ2],[Wei],[CL],[CTZ]} and references therein.

 We now state our main theorems.  We firstly invoke the Hardy-Sobolev inequality \eqref{anHS} to generalize Yu's regularity class in \cite{[Yu]}. Our results can be formulated as
\begin{theorem}\label{the1.4}
	Suppose that $u$ be an axisymmetric weak solution of the Navier-Stokes system \eqref{ANS}
	associated with the axisymmetric divergence-free initial data $u_0 \in H^2(\mathbb{R}^3)$. If
	\be\label{WWY1}\ba
	&|x_{3}|^{\alpha}u_{\theta}\in L^{q}(0,T; L^{p}(\mathbb{R}^{3})),\\
	& \f2q+\f3p=1-\alpha, 0\leq \alpha<\f14, \f{3}{1-\alpha}<p\leq\infty, \f{2}{1-\alpha}\leq q<\infty.
	\ea\ee
	or
	$u_{\theta} |x_{3}|^{\alpha}\in L^{\infty}(0,T;L^{\f{3}{1-\alpha}}(\mathbb{R}^{3}))$ and  the norm of     $\|u_{\theta} |x_{3}|^{\alpha}\|_{L^{\infty}(0,T;L^{\f{3}{1-\alpha}}(\mathbb{R}^{3}))} $ is sufficiently small,
	then $u$ is smooth in $(0, T] \times\mathbb{R}^3$.
\end{theorem}

Besides, the Hardy-Sobolev inequality \eqref{anHS1}  allows us to slightly  improve the result by Chen,   Fang and   Zhang \cite{[CFZ1]}. The corresponding results can be stated as follows.
\begin{theorem}\label{the1.5}
	Let $u$ be an axisymmetric weak solution of the Navier-Stokes system \eqref{ANS}
	associated with the axisymmetric initial data $u_0 \in H^2(\mathbb{R}^3)$ satisfying div$u_0=0$. If \be\label{WWY2}
	r^{\alpha}u_{\theta}\in L^{q}(0,T; L_{x_{3}}^{p_{3},\infty}(\mathbb{R}  )L_{x_{2}}^{p_{2},\infty}(\mathbb{R}  )L_{x_{1}}^{p_{1},\infty}(\mathbb{R}  )), \f2q+\f1p_{3}+\f1p_{2}+\f1p_{1}=1-\alpha,\ee$$ 0\leq \alpha<\f12, \f{3}{1-\alpha}<p_{i}\leq\infty, \f{2}{1-\alpha}\leq q<\infty.
	$$
	or $r^{\alpha}u_{\theta}\in L^{\infty}(0,T; L_{x_{3}}^{p_{3},\infty}(\mathbb{R}  )L_{x_{2}}^{p_{2},\infty}(\mathbb{R}  )L_{x_{1}}^{p_{1},\infty}(\mathbb{R}  ))$ with $ \f1p_{3}+\f1p_{2}+\f1p_{1}=1-\alpha $  and the sufficiently small norm of
	$\|r^{\alpha}u_{\theta}\|_{L^{\infty}(0,T; L_{x_{3}}^{p_{3},\infty}(\mathbb{R}  )L_{x_{2}}^{p_{2},\infty}(\mathbb{R}  )L_{x_{1}}^{p_{1},\infty}(\mathbb{R}  ))},$
	then $u$ is smooth in $(0, T] \times\mathbb{R}^3$.
\end{theorem}
\begin{remark}
	In view of  inclusion relationship $L^{p_{3}}L^{p_{2}}L^{p_{1}}(\mathbb{R}^{3})\subset L^{p_{3},\infty}L^{p_{2},\infty}L^{p_{1},\infty}(\mathbb{R}^{3})$ and $L^{p_{1},\infty}(\mathbb{R}^{3})\subset L^{p_{3},\infty}L^{p_{2},\infty}L^{p_{1},\infty}(\mathbb{R}^{3})$ , we know that the condition
	$$
	r^{\alpha}u_{\theta}\in L^{q}(0,T;L^{\overrightarrow{q}}(\mathbb{R}^{3})),
	\f2q+\f1p_{3}+\f1p_{2}+\f1p_{1}=1-\alpha, 0\leq \alpha<\f12, \f{3}{1-\alpha}<p_{i}\leq\infty, \f{2}{1-\alpha}\leq q\leq\infty,
	$$
	or
	$$r^{d}u_{\theta}\in L^{q}(0,T; L^{p_{1},\infty}(\mathbb{R}^{3})), \f2q+\f3p_{1}=1-\alpha, 0\leq \alpha<\f12, \f{3}{1-d}<p_{1}\leq\infty, \f{2}{1-\alpha}\leq q\leq\infty,$$
	ensures  the   non-breakdown of solutions of the Navier-Stokes equations,   which is also a generalization of  \eqref{CFZ}.
\end{remark}\begin{remark}
	As Theorem \ref{the1.5}, one can  generalize \eqref{cfz2} to the  anisotropic  Lorentz spaces. We leave this to the interesting readers.
\end{remark}

In the spirit of Theorem \ref{the1.5}, we can show similar result involving radial component $r^{\alpha}u_{r}$.
\begin{theorem}\label{the1.6}
Let $u$ be an axisymmetric weak solution of the Navier-Stokes system \eqref{ANS}
associated with the axisymmetric initial data $u_0 \in H^2(\mathbb{R}^3)$ satisfying div$u_0=0$. If
\be\label{WWY3}\ba
&r^{\alpha}u_{r}\in L^{q}(0,T; L^{\overrightarrow{p},\infty}(\mathbb{R}^{3})),
\f2q+\f1p_{3}+\f1p_{2}+\f1p_{1}=1-\alpha, \\&0\leq \alpha<\f12, \f{3}{1-\alpha}<p_{i}\leq\infty, \f{2}{1-\alpha}\leq q\leq\infty, \ea\ee
 or  $r^{\alpha}u^{\theta}\in L^{\infty}(0,T; L_{x_{3}}^{p_{3},\infty}(\mathbb{R}  )L_{x_{2}}^{p_{2},\infty}(\mathbb{R}  )L_{x_{1}}^{p_{1},\infty}(\mathbb{R}  ))$ with $ \f1p_{3}+\f1p_{2}+\f1p_{1}=1-\alpha $  and its norm is sufficiently small,
then $u$ is smooth in $(0, T] \times\mathbb{R}^3$.
\end{theorem}
\begin{remark}
This theorem is an improvement of corresponding results in  \cite{[KPZ]}.
\end{remark}
This  paper is organized as follows. In   section 2,
  we list some basic  fact of  the  various  functions spaces used in this paper and the auxiliary lemma for the discussion of the axisymmetric  Navier-Stokes equations.
  The section 3  is devoted to the proof of     Hardy-Sobolev type inequalities.
 In  Section 4, we are concerned with the sufficient regularity  conditions for  weak solutions
 of the 3D axisymmetric Navier-Stokes equations by  applications of the Hardy-Sobolev type inequalities established in last section.

\section{ Function spaces and key auxiliary lemmas} \label{section2}
For $p\in [1,\,\infty]$, the notation $L^{p}(0,\,T;X)$ stands for the set of measurable function $f$ on the interval $(0,\,T)$ with values in $X$ and $\|f\|_{X}$ belonging to $L^{p}(0,\,T)$.  The Fourier transform $\hat{f}$ of a Schwartz function $f$ on $\mathbb{R}^{n}$ is defined as $\hat{f}(\xi)=\frac{1}{(2 \pi)^{\frac{n}{2}}}\int_{\mathbb{R}^{n}}f (x)e^{- i\xi\cdot x}\,dx, $
and the inverse Fourier transform $f^{\vee}$ is given by
$
f^{\vee}(\xi)=\widehat{f}(-\xi)
$
for all $\xi \in \mathbb{R}^{n}.$ Furthermore, for $s\geq 0,$ we define $\Lambda^{s}f $ by
$\widehat{\Lambda^{s} f}(\xi)=(\sum^{n}_{i=1}|\xi_{i}|^{2})^{s/2}\hat{f}(\xi)$, where the notation $\Lambda$ stands for the square root of negative Laplacian $(-\Delta)^{1/2}$.  Similarly,  we denote  $\widehat{\Lambda_{x_{i}}^{ s} f}(\xi)=|\xi_{i}|^{ s}\hat{f}(\xi)$ and
$\widehat{\Lambda_{x_{1}, \cdots,x_{k}}^{ s} f}(\xi)=(\sum^{k}_{i=1}|\xi_{i}|^{2})^{ s/2}\hat{f}(\xi)$.
We   will use $C$ to denote an absolute
   constant which may be different from line to line unless otherwise stated.

 Next, we recall the definition of   Lorentz spaces.
  Denote
  the distribution function of a   measurable function  $f$ on $\mathbb{R}^{n}$  by
  $f_{\ast}$ defined on $[0,\infty)$ by
$$
f_{\ast}(\alpha)=|\{x\in \mathbb{R}^{n}:|f(x)|>\alpha\}|.
$$
The decreasing rearrangement of $f$ is the function $f^{\ast}$ defined on $[0,\infty)$ by
$$
f^{\ast}(t)=\inf\{\alpha>0: f_{\ast}(\alpha)\leq t\}.
$$
For $p,q\in (0,\infty]$, we write
$$
\|f\|_{L^{p,q}(\mathbb{R}^{n})}=\left\| t^{\f{1}{p}}f^{\ast}(t) \right\|_{L^{q}\left(\mathbb{R}^{+}, \frac{d t}{t}\right)} =\left\{\ba
&\B(\int_{0}^{\infty}\left(t^{\f{1}{p}}f^{\ast}(t)\right)^{q}\f{dt}{t}\B)^{\f1q}, ~~\text { if } q<\infty, \\
 &\sup_{t>0}t^{\f{1}{p}}f^{\ast}(t), ~~\text { if } q=\infty.
\ea\right.
$$
Furthermore,
$$
L^{p,q}(\mathbb{R}^{n})=\big\{f: f~ \text{is a measurable function on}~ \mathbb{R}^{n} ~\text{and} ~\|f\|_{L^{p,q}(\mathbb{R}^{n})}<\infty\big\}.
$$
It is well-known  that $L^{\infty,\infty}=L^{\infty}$, $L^{q,q}=L^{q}$ and $L^{\infty,q}=\{0\}$	for $0<q<\infty$.

The study    of      anisotropic Lebesgue spaces $L^{\overrightarrow{q}} (\mathbb{R}^{n})$  originated from
  Benedek-Panzone's work \cite{[BP]}.
A function $f$ belongs to  the anisotropic Lebesgue space $L^{\overrightarrow{q}} (\mathbb{R}^{n})$
if
  $$\ba\|f\|_{L^{\overrightarrow{q}}_{x}(\mathbb{R}^{n})}=&\|f\|_{L_{1}^{q_{1}}L_{2}^{q_{2}}\cdots L_{n}^{q_{n}}(\mathbb{R}^{n})}\\=&
  \B\|\cdots\big\|\|f\|_{L_{x_1}^{q_{1}}(\mathbb{R})}\big\|_{L_{x_2}^{q_{2}}(\mathbb{R})}\cdots\B\|_{L_{x_n}^{q_{n}} (\mathbb{R})}<\infty.  \ea$$
Mixed Lorentz space $L^{\overrightarrow{p},\overrightarrow{q}}(\mathbb{R}^{n})$ was introduced in  \cite{[Blozinski],[KT],[Fernandez]} and its norm is
determined by
  $$\ba
  \|f\|_{L^{\overrightarrow{p},\overrightarrow{q}}_{x}(\mathbb{R}^{n})}=&\|f\|_{L_{1}^{p_{1},q_{1}}L_{2}^{p_{2},q_{2}}\cdots L_{n}^{p_{n},q_{n}}(\mathbb{R}^{n})}\\=&
  \B\|\cdots\big\|\|f\|_{L_{x_1}^{p_{1},q_{1}}(\mathbb{R})}\big\|_{L_{x_2}^{p_{2},q_{2}}(\mathbb{R})}\cdots\B\|_{L_{x_n}^{p_{n},q_{n}} (\mathbb{R})}<\infty. \ea $$

For the convenience of readers,
we present some properties of mixed Lorentz spaces which will be frequently used in this paper as follows.
\begin{itemize}
\item
H\"older's inequality in mixed Lorentz spaces  \cite{[Blozinski],[KT]}
\be\label{HolderIQ}
\ba &\|fg\|_{L^{\overrightarrow{r},\overrightarrow{s}}(\mathbb{R}^{n})}\leq C \,\|f\|_{L^{\overrightarrow{r_{1}},\overrightarrow{s_{1}}}(\mathbb{R}^{n})}
\|g\|_{L^{\overrightarrow{r_{2}},\overrightarrow{s_{2}}}(\mathbb{R}^{n})},
\\
&\text{with}~~\f{1}{\overrightarrow{r}}=\f{1}{\overrightarrow{r_{1}}}+\f{1}{\overrightarrow{r_{2}}},~~\f{1}{\overrightarrow{s}}=\f{1}{\overrightarrow{s_{1}}}+\f{1}{\overrightarrow{s_{2}}},
~~0<\overrightarrow{r_{1}},\overrightarrow{r_{2}},
\overrightarrow{s_{1}},\overrightarrow{s_{2}}\leq \infty.\ea\ee

\item

The mixed Lorentz spaces increase as the exponent $\overrightarrow{q}$ increases \cite{[Blozinski],[KT]}

For $0< \overrightarrow{p}\leq\infty$ and $0< \overrightarrow{q_{1}}<\overrightarrow{q_{2}}\leq\infty,$
\be\label{lincreases}
\|f\|_{L^{\overrightarrow{p},\overrightarrow{q_{2}}}(\mathbb{R}^{n})}\leq C\|f\|_{L^{\overrightarrow{p},\overrightarrow{q_{1}}}(\mathbb{R}^{n})}.
\ee

\item Sobolev inequality in mixed Lorentz spaces \cite{[Blozinski],[KT]}, for $1\leq\overrightarrow{\ell}\leq\infty$
\be\label{sl}
\|f\|_{L^{\overrightarrow{p}, \overrightarrow{\ell}}(\mathbb{R}^{n})}\leq C \|(-\Delta)^{\f{s}{2}} f\|_{L^{\overrightarrow{r},\overrightarrow{\ell}}(\mathbb{R}^{n})},~~ \text{with}~~ \sum_{i=1}^{n}\left(\frac{1}{r_{i}}-\frac{1}{p_{i}}\right)=s,~1<r_{i}<p_{i}<\infty.
\ee
\item Young inequality in mixed Lorentz spaces \cite{[Blozinski],[KT]}

Let $1<\overrightarrow{p},\overrightarrow{q},\overrightarrow{r}<\infty$, $0<\overrightarrow{s_{1}},\overrightarrow{s_{2}}\leq\infty$
,$\f{1}{\overrightarrow{p}}+\f{1}{\overrightarrow{q}}=\f{1}{\overrightarrow{r}}+1$, and $ \f{1}{\overrightarrow{s}}=\f{1}{\overrightarrow{s_{1}}}+\f{1}{\overrightarrow{s_{2}}}$. Then there holds
\be\label{young}
\|f\ast g\|_{L^{\overrightarrow{r},\overrightarrow{s}}(\mathbb{R}^{n})}\leq C \,\|f \|_{L^{\overrightarrow{p},\overrightarrow{s_{1}}}(\mathbb{R}^{n})}\|g \|_{L^{\overrightarrow{q},\overrightarrow{s_{2}}}(\mathbb{R}^{n})}.
\ee
\end{itemize}
Finally,  we recall two known results involving the regularity of the 3D axisymmetric  Navier-Stokes equations.
\begin{lemma}(\cite{[MZ],[CFZ]})\label{magic} Assume $u$ is the smooth axisymmetric solution of equations \eqref{ANS}, $\omega=curl\,u,$ then there holds
	\begin{equation}\label{equality}
	\frac{u_r}{r}=\partial_{x_{3}} \Delta^{-1}\Gamma-
2\frac{\partial_r}{r}\Delta^{-1}\partial_{x_{3}}\Delta^{-1}
\Gamma.
	\end{equation}
	In addition, for  $1< p<\infty, $ it is valid that
	\begin{equation}\label{R-1}
	\norm{\nabla \frac{u_r}{r}}_{L^p}\leq C\norm{\Gamma}_{L^p},
	\end{equation}
	and
	\begin{equation}\label{R-2}
	\norm{\nabla ^2\frac{u_r}{r}}_{L^p}\leq C\norm{\partial_{x_3}\Gamma}_{L^p}.
\end{equation}\label{lem2.1}
\end{lemma}

\begin{lemma} (\cite{[CFZ]})\label{lem2.2}
Let $u\in C([0,T);H^{2}(\mathbb{R}^{3})) \cap L^{2}([0,T);H^{2}(\mathbb{R}^{3})) $ be the axisymmetric solution of the Navier-Stokes equations with the axisymmetric   divergence-free  initial data
$u_{0}$. If $T<\infty$ and $\Gamma \in  L^{\infty}([0,T);L^{2}(\mathbb{R}^{3}))$, then $u$ can be continued beyond $T$.
\end{lemma}

 \section{Proof of anisotropic Hardy-Sobolev inequality}
This section is devoted to the proof of various Hardy-Sobolev type inequalities, in which the main tool is the mixed Lorentz spaces. Our critical observation is that  the function  $[\prod\limits^{i}_{j=1}(\prod\limits_{\ell=1}^{k_{j}} |x_{\sum^{j-1}_{m=1}k_{m}+\ell}|^{\f{\alpha_{j}}{k_{j}}})]^{-1}$
 belongs to  anisotropic   Lorentz  space $L^{\f{k_{1}}{\alpha_{1}},\infty}(\mathbb{R}^{k_{1}})\cdots
L^{\f{k_{i}}{\alpha_{i}},\infty}(\mathbb{R}^{k_{i}})$. To  illustrate the above argument, we begin with the proof of Theorem \ref{the1.1}.
 \begin{proof} [Proof of Theorem \ref{the1.1}]
Notice that $|x|^{-s}\in L^{\f{n}{ s},\infty}(\mathbb{R}^{n})$.
Combing this with
  the H\"older inequality  \eqref{HolderIQ} in  anisotropic  Lorentz spaces, we infer that
$$\ba
\B\|\f{f(x)}{|x|^{s}}\B\|_{L^{\overrightarrow{p} ,\overrightarrow{q}}(\mathbb{R}^{n})}\leq
C\||x|^{-s}\|_{L^{\f{n}{ s},\infty} (\mathbb{R}^{n})}\|f\|_{L^{\overrightarrow{p^{\ast}} ,\overrightarrow{q} }(\mathbb{R}^{n})}
\leq  C\|f\|_{L^{\overrightarrow{p^{\ast}} ,\overrightarrow{q} }(\mathbb{R}^{n})},
\ea$$
where  $ p_{i}^{\ast}=\f{np_{i}}{n-sp_{i}}$ for $1\leq i\leq n$.

The Sobolev inequality \eqref{sl} in mixed Lorentz spaces  further guarantees that
$$\ba
\|f\|_{L^{\overrightarrow{p^{\ast}} ,\overrightarrow{q} }(\mathbb{R}^{n})}
\leq  C
 \|\Lambda^{s}f\|_{L^{\overrightarrow{p},\overrightarrow{q}} (\mathbb{R}^{n})}.
\ea$$
The proof of this theorem is finished.
\end{proof}
Now  we turn our attention to the proof of Theorem \ref{the1.2}.
\begin{proof} [Proof of Theorem \ref{the1.2}]
Thanks to the celebrated Fubini's theorem for the $L^{p}(\mathbb{R}^{n})$ norm, it suffices to prove \eqref{anHS} for the case that $i_{m}=m$ with $1\leq m\leq \Sigma_{j=1}^{i}k_{j}$.
According to the definition of $r_{j}$, we see that, for $1\leq j\leq i$ and $1\leq\ell\leq k_{j}$,
$$r_{j}\geq |x_{\sum^{j-1}_{m=1}k_{m}+\ell}|.$$
It turns out that
$$r_{j}^{\alpha_{j}}=(r_{j}^{\f{\alpha_{j}}{k_{j}}})^{k_{j}}
\geq \prod_{\ell=1}^{k_{j}} |x_{\sum^{j-1}_{m=1}k_{m}+\ell}|^{\f{\alpha_{j}}{k_{j}}}.$$
Hence, we see that
$$
\B\|\f{f(x_{1},x_{2},\cdots,x_{n})}{\prod^{i}_{j=1}|r_{j}|^{\alpha_{j}}}
\B\|_{L^{p}(\prod^{i}_{j=1}\mathbb{R}^{k_{j}})}
\leq \B\|\f{f(x_{1},x_{2},\cdots,x_{n})}{\prod^{i}_{j=1}\B(\prod_{\ell=1}^{k_{j}} |x_{\sum^{j-1}_{m=1}k_{m}+\ell}|^{\f{\alpha_{j}}{k_{j}}}\B)}
\B\|_{L^{p}(\prod^{i}_{j=1}\mathbb{R}^{k_{j}})}.$$
It is clear that
$$
\B\|\f{1}{\prod^{i}_{j=1}\B(\prod_{\ell=1}^{k_{j}} |x_{\sum^{j-1}_{m=1}k_{m}+\ell}|^{\f{\alpha_{j}}{k_{j}}}\B)}
\B\|_{L^{\f{k_{1}}{\alpha_{1}},\infty}(\mathbb{R}^{k_{1}})\cdots
L^{\f{k_{i}}{\alpha_{i}},\infty}(\mathbb{R}^{k_{i}})}<\infty.
$$
By the H\"older inequality \eqref{HolderIQ} and Sobolev inequality  \eqref{sl} in mixed Lorentz spaces, we observe  that
$$\ba
&\B\|\f{f(x_{1},x_{2},\cdots,x_{n})}{\prod^{i}_{j=1}\B(\prod_{\ell=1}^{k_{j}} |x_{\sum^{j-1}_{m=1}k_{m}+\ell}|^{\f{\alpha_{j}}{k_{j}}}\B)}
\B\|_{L^{p}(\prod^{i}_{j=1}\mathbb{R}^{k_{j}})}\\\leq&
C\B\|\f{1}{\prod^{i}_{j=1}\B(\prod_{\ell=1}^{k_{j}} |x_{\sum^{j-1}_{m=1}k_{m}+\ell}|^{\f{\alpha_{j}}{k_{j}}}\B)}
\B\|_{L^{\f{k_{1}}{\alpha_{1}},\infty}(\mathbb{R}^{k_{1}})\cdots
L^{\f{k_{i}}{\alpha_{i}},\infty}(\mathbb{R}^{k_{i}})}
\|f\|_{L^{\f{pk_{1}}{k_{1}-p\alpha_{1}},p}(\mathbb{R}^{k_{1}})\cdots
L^{\f{p k_{i}}{k_{i}-p\alpha_{i}},p}(\mathbb{R}^{k_{i}})}
\\\leq&C \|\Lambda_{x_{1},\cdots,x_{\Sigma_{j=1}^{i}k_{j}}}^{\alpha_{1}
+\cdots+\alpha_{i}}f\|_{L^{p}(\prod^{i}_{j=1}\mathbb{R}^{k_{j}})}. \ea$$
It follows that
$$\ba
&\B\|\f{f(x_{1},x_{2},\cdots,x_{n})}{\prod^{i}_{j=1}\B(\prod_{\ell=1}^{k_{j}} |x_{\sum^{j-1}_{m=1}k_{m}+\ell}|^{\f{\alpha_{j}}{k_{j}}}\B)}
\B\|_{L^{p}(\mathbb{R}^{n})} 
\leq  C\|\Lambda_{x_{1},\cdots,x_{\Sigma_{j=1}^{i}k_{j}}}^{\alpha_{1}
+\cdots+\alpha_{i}}f\|_{L^{p}(\mathbb{R}^{n})}.
\ea$$
Consequently, we complete the proof of this theorem.
\end{proof}

\begin{proof} [Proof of Corollary \ref{coro}]
It suffices  to notice that
$$\ba
&\B\|\f{f(x_{1},x_{2},\cdots,x_{n})}{\prod^{i}_{j=1}\B(\prod_{\ell=1}^{k_{j}} |x_{\sum^{j-1}_{m=1}k_{m}+\ell}|^{\f{\alpha_{j}}{k_{j}}}\B)}
\B\|_{L^{\overrightarrow{p_{1}},\overrightarrow{q_{1}}}
(\mathbb{R}^{  k_{1}})\cdots L^{\overrightarrow{p_{i}},\overrightarrow{q_{i}}}(\mathbb{R}^{ k_{i}})}\\\leq&C
\B\|\f{1}{\prod^{i}_{j=1}\B(\prod_{\ell=1}^{k_{j}} |x_{\sum^{j-1}_{m=1}k_{m}+\ell}|^{\f{\alpha_{j}}{k_{j}}}\B)}
\B\|_{L^{\f{k_{1}}{\alpha_{1}},\infty}(\mathbb{R}^{k_{1}})\cdots
 L^{\f{k_{i}}{\alpha_{i}},\infty}(\mathbb{R}^{k_{i}})}\|f\|_{L^{\overrightarrow{p_{1}^{\ast}},\overrightarrow{q_{1}} }(\mathbb{R}^{k_{1}})\cdots
L^{\overrightarrow{p_{i}^{\ast}},\overrightarrow{q_{i}}}(\mathbb{R}^{k_{i}})}, \ea$$
where $\frac{1}{\overrightarrow{p_{j}^{\ast}}}=\frac{1}{\overrightarrow{p_{j}}}-\frac{\alpha_{j}}{k_{j}}$ for $1\leq j\leq i$.
Arguing in the same manner as above,
we can achieve the proof of this  corollary. We leave this to the interested readers.
\end{proof}
  \section{An
application of Hardy-Sobolev   inequality to the  axisymmetric  Navier-Stokes equations}
 In this section, we will show that the  anisotropic Hardy-Sobolev inequality derived in the last section is useful for studying the solutions of  3D  axisymmetric  Navier-Stokes equations.
 \begin{proof}[Proof of Theorem \ref{the1.4}]
For  an axisymmetric function $f$,  there holds $|\partial_{r}f|^{2}+|\partial_{x_{3}}f|^{2}=|\nabla f|^{2}$.
Hence, taking the inner product of $\Phi$ equation in \eqref{w} with $\Phi$ and integrating by parts yields
\begin{equation}\label{4.2}\begin{split}
\frac12\frac{d}{dt}\|\Phi\|^{2}_{L^{2}(\mathbb{R}^3)}+\|\nabla \Phi\|_{L^{2}(\mathbb{R}^3)}^2 &=\int_{\mathbb{R}^3} \Phi (\omega_{r}\partial_{r}+\omega_{z}\partial_{x_3})\frac{u_{r}}{r} dx\\
&=\int_{\mathbb{R}^3} u_{\theta}(\partial_{r}\frac{u_{r}}{r} \partial_{x_3}\Phi-\partial_{x_3}\frac{u_{r}}{r} \partial_{r}\Phi)   dx.
\end{split}
\end{equation}
Similarly,
\begin{equation} \label{4.1}
\frac12\frac{d}{dt}\|\Gamma\|^{2}_{L^{2}(\mathbb{R}^3)}+\|\nabla \Gamma\|_{L^{2}(\mathbb{R}^3)}^2 =-2\int_{\mathbb{R}^3} u_{\theta} \frac{\Gamma}{r} \Phi  dx.
\end{equation}
We estimate the terms on the right hand side of \eqref{4.2} and \eqref{4.1} in the following two cases respectively.
\\(1) If
\be\ba
&|x_{3}|^{\alpha}u_{\theta}\in L^{q}(0,T; L^{p}(\mathbb{R}^{3})),\\
& \f2q+\f3p=1-\alpha, 0\leq \alpha<\f14, \f{3}{1-\alpha}<p\leq\infty, \f{2}{1-\alpha}\leq q<\infty.
\ea\ee
For the first term on the right hand of \eqref{4.2}, it follows from the H\"older inequality that
\be\ba\label{4.3}
\B|\int_{\mathbb{R}^3} u_{\theta} \partial_{r}\frac{u_{r}}{r} \partial_{x_3}\Phi dx\B|
=& \B|\int_{\mathbb{R}^3} u_{\theta}|x_{3}|^{\alpha} \f{\partial_{r}\frac{u_{r}}{r} }{|x_{i}|^{ \alpha } }\partial_{x_3}\Phi dx\B|\\
\leq&\|u_{\theta}|x_{3}|^{\alpha} \|_{L^{p}(\mathbb{R}^3)}  \B\|  \f{\partial_{r}\frac{u_{r}}{r} }{|x_{3}|^{ \alpha } } \B\|_{L^{\f{2p}{p-2}}(\mathbb{R}^3)} \|\partial_{x_3}\Phi\|_{L^{2}(\mathbb{R}^3)}.
 \ea\ee
Since  $p>\f3{1-\alpha}$ and $\alpha<\f14$ ensure  $\alpha<\f{p-2}{2p},$
the Hardy inequality in Theorem \ref{the1.2} can be applied
 $$
  \B\|  \f{\partial_{r}\frac{u_{r}}{r} }{|x_{3}|^{ \alpha } } \B\|_{L^{\f{2p}{p-2}}(\mathbb{R}^3)}\leq C
 \B\| \Lambda^{\alpha}\partial_{r}\frac{u_{r}}{r}   \B\|_{L^{\f{2p}{p-2}}(\mathbb{R}^3)}.
 $$
Thanks to the fractional Gagliardo-Nirenberg inequalities   (see, e.g.\cite{[WXW],[WWY]}   and references therein) and  \eqref{R-1}, we infer that
$$
  \B\| \Lambda^{\alpha}\partial_{r}\frac{u_{r}}{r}   \B\|_{L^{\f{2p}{p-2}}(\mathbb{R}^3)} \leq C
 \B \|  \partial_{r}\frac{u_{r}}{r}   \B\|_{L^{2}(\mathbb{R}^3)}^{\f{p-3-p\alpha}{p}}\B\| \nabla\partial_{r}\frac{u_{r}}{r}   \B\|^{\f{3+p\alpha}{p}}_{L^{2}(\mathbb{R}^3)}
\leq C
 \| \Gamma   \|_{L^{2}(\mathbb{R}^3)}^{\f{p-3-p\alpha}{p}} \| \nabla\Gamma  \|^{\f{3+p\alpha}{p}}_{L^{2}(\mathbb{R}^3)}.
 $$
Therefore,
\be\label{4.4}
\B\|  \f{\partial_{r}\frac{u_{r}}{r} }{|x_{3}|^{ \alpha } } \B\|_{L^{\f{2p}{p-2}}(\mathbb{R}^3)}\leq C \| \Gamma   \|_{L^{2}(\mathbb{R}^3)}^{\f{p-3-p\alpha}{p}} \| \nabla\Gamma  \|^{\f{3+p\alpha}{p}}_{L^{2}(\mathbb{R}^3)}.
\ee
Inserting \eqref{4.4} into \eqref{4.3} and using the Young inequality yield
\be\ba\label{k1}
\B| \int_{\mathbb{R}^3} u_{\theta} \partial_{r}\frac{u_{r}}{r} \partial_{x_3}\Phi dx\B|
 \leq&C\|u_{\theta}|x_{3}|^{\alpha} \|^{\frac{2p}{p-p\alpha-3}}_{L^{p}(\mathbb{R}^3)}  \| \Gamma   \|^2_{L^{2}(\mathbb{R}^3)}+\frac14 (\| \nabla\Gamma  \|^2_{L^{2}(\mathbb{R}^3)}+\|\nabla\Phi\|^2_{L^{2}(\mathbb{R}^3)}).
  \ea\ee
  Likewise,
 \be\label{k2}
\B| \int_{\mathbb{R}^3} u_{\theta } \partial _{z}\frac {u_{r}}{r} \partial _{r}\Phi dx\B|  \leq C\|u_{\theta}|x_{3}|^{\alpha} \|^{\frac{2p}{p-p\alpha-3}}_{L^{p}(\mathbb{R}^3)}  \| \Gamma   \|^2_{L^{2}(\mathbb{R}^3)}+\frac14 (\| \nabla\Gamma  \|^2_{L^{2}(\mathbb{R}^3)}+\|\nabla\Phi\|^2_{L^{2}(\mathbb{R}^3)}).
\ee
It remains to bound the term on the right hand side of \eqref{4.1}.  To this end, using the H\"older inequality, we conclude  that
 $$\ba
 \B|\int_{\mathbb{R}^3} u_{\theta} \frac{\Gamma}{r} \Phi dx\B|=&\B| \int_{\mathbb{R}^3} u_{\theta} |x_{3}|^{\alpha} \f{\Gamma}{ |x_{3}|^{\f\alpha2}r^{\f12}}\f{\Phi}{ |x_{3}|^{\f\alpha2}r^{\f12}}dx\B|\\
\leq&\|u_{\theta}|x_{3}|^{\alpha} \|_{L^{p}(\mathbb{R}^3)}  \B\| \f{\Gamma}{ |x_{3}|^{\f\alpha2}r^{\f12}}\B\|_{L^{\f{2p}{p-1}}(\mathbb{R}^3)}\B\| \f{\Phi}{ |x_{3}|^{\f\alpha2}r^{\f12}}\B\|_{L^{\f{2p}{p-1}}(\mathbb{R}^3)}.
 \ea
 $$
 For $p>\frac{3}{1-\alpha}$ and $\alpha<\frac14,$ using  Hardy-Sobolev inequality  \eqref{anHS} and  the Gagliardo-Nirenberg inequality again yield
 $$\B\| \f{\Gamma}{ |x_{3}|^{\f\alpha2}r^{\f12}}\B\|_{L^{\f{2p}{p-1}}(\mathbb{R}^3)}\leq C \|\Lambda^{\f{1+\alpha}{2}}\Gamma\|_{L^{\f{2p}{p-1}}(\mathbb{R}^3)}\leq C\|\Gamma\|_{L^{2}(\mathbb{R}^3)} ^{\f{p-3-p\alpha}{2p}}\|\nabla\Gamma\|_{L^{2}(\mathbb{R}^3)}^{\f{p+3+p\alpha}{2p}}.
 $$
 Similarly,
 $$\B\| \f{\Phi}{ |x_{3}|^{\f\alpha2}r^{\f12}}\B\|_{L^{\f{2p}{p-1}}(\mathbb{R}^3)}\leq C \|\Lambda^{\f{1+\alpha}{2}}\Phi\|_{L^{\f{2p}{p-1}}(\mathbb{R}^3)}\leq C\|\Phi\|_{L^{2}(\mathbb{R}^3)} ^{\f{p-3-p\alpha}{2p}}\|\nabla\Phi\|_{L^{2}(\mathbb{R}^3)}^{\f{p+3+p\alpha}{2p}}.
 $$
Hence, by using the Young inequality, we arrive at
\be\ba\label{4.10}
&\B|\int_{\mathbb{R}^3} u_{\theta} \frac{\Gamma}{r} \Phi dx  \B|\\\leq & C\|u_{\theta}|x_{3}|^{\alpha} \|^{\frac{2p}{p-p\alpha-3}}_{L^{p}(\mathbb{R}^3)} (\| \Phi\|^2_{L^{2}(\mathbb{R}^3)}+\| \Gamma\|^2_{L^{2}(\mathbb{R}^3)}) +\frac14(\|\nabla\Phi\|^2_{L^{2}(\mathbb{R}^3)}+\|\nabla\Gamma\|^2_{L^{2}
(\mathbb{R}^3)}).
\ea\ee
 Combining   \eqref{4.2}  and \eqref{4.1} with \eqref{k1}, \eqref{k2} and \eqref{4.10}, we find
$$ \ba
&\frac{d}{dt}\big(\|\Phi\|^{2}_{L^{2}(\mathbb{R}^3)}+\|\Gamma\|^{2}_{L^{2}(\mathbb{R}^3)}\big)+\big(\|\nabla \Phi\|_{L^{2}(\mathbb{R}^3)}^2+\|\nabla \Gamma\|_{L^{2}(\mathbb{R}^3)}^2\big) \\ \leq&C\|u_{\theta}|x_{3}|^{\alpha} \|^{\frac{2p}{p-p\alpha-3}}_{L^{p}(\mathbb{R}^3)} (\| \Phi\|^2_{L^{2}(\mathbb{R}^3)}+\| \Gamma\|^2_{L^{2}(\mathbb{R}^3)}).
\ea$$
The Gronwall lemma and \eqref{WWY1}  enables us to  obtain
\begin{equation*}
\|\Gamma\|_{L^\infty(0,T;L^2) }\leq C<\infty.
\end{equation*}
(2) If $u_{\theta} |x_{3}|^{\alpha}\in L^{\infty}(0,T;L^{\f{3}{1-\alpha}}(\mathbb{R}^{3}))$ and  the norm of     $\|u_{\theta} |x_{3}|^{\alpha}\|_{L^{\infty}(0,T;L^{\f{3}{1-\alpha}}(\mathbb{R}^{3}))} $ is sufficiently small.
\\ Similarly  to \eqref{4.3}, we know that
 $$\ba
\B| \int_{\mathbb{R}^3} u_{\theta} \partial_{r}\frac{u_{r}}{r} \partial_{x_3}\Phi dx\B|
\leq&\|u_{\theta}|x_{3}|^{\alpha} \|_{L^{\f{3}{1-\alpha}}(\mathbb{R}^3)}  \B\|  \f{\partial_{r}\frac{u_{r}}{r} }{|x_{3}|^{ \alpha } } \B\|_{L^{\f{6}{1+2\alpha}}(\mathbb{R}^3)} \|\partial_{x_3}\Phi\|_{L^{2}(\mathbb{R}^3)}.\ea$$
Since $\alpha<\f14,$  it follows that  $\alpha<\f{6}{1+2\alpha}$. Then, we can invoke  Hardy inequality \eqref{anHS} and the Sobolev inequality to obtain
$$\B\|  \f{\partial_{r}\frac{u_{r}}{r} }{|x_{3}|^{ \alpha } } \B\|_{L^{\f{6}{1+2\alpha}}(\mathbb{R}^3)} \leq C\B\|  \Lambda^{\alpha}\partial_{r}\frac{u_{r}}{r}   \B\|_{L^{\f{6}{1+2\alpha}}(\mathbb{R}^3)}\leq C\|\nabla\partial_{r}\frac{u_{r}}{r} \|_{L^{2}(\mathbb{R}^3)}\leq C\|\nabla\Gamma \|_{L^{2}(\mathbb{R}^3)},$$
which turns out that
\be\label{k11}
 \B|\int_{\mathbb{R}^3} u_{\theta} \partial_{r}\frac{u_{r}}{r} \partial_{x_3}\Phi \B|
\leq C
 \|u_{\theta} |x_{3}|^{\alpha}\|_{L^{\f{3}{1-\alpha}}(\mathbb{R}^3)}
 (\|\nabla\Phi\|^2_{L^{2}(\mathbb{R}^3)}+\|\nabla\Gamma\|^2_{L^{2}(\mathbb{R}^3)}).
\ee
In   the same manner as above, there holds
\be\label{k12}
 \B|\int_{\mathbb{R}^3} u_{\theta } \partial _{x_{3}}\frac {u_{r}}{r} \partial _{r}\Phi \B|\leq  C
 \|u_{\theta} |x_{3}|^{\alpha}\|_{L^{\f{3}{1-\alpha}}(\mathbb{R}^3)}
 (\|\nabla\Phi\|^2_{L^{2}(\mathbb{R}^3)}+\|\nabla\Gamma\|^2_{L^{2}(\mathbb{R}^3)}).
 \ee
 Using the H\"older inequality, we see that
$$\ba
 \B|\int_{\mathbb{R}^3} u_{\theta} \frac{\Gamma}{r} \Phi \B|
\leq&\|u_{\theta}|x_{3}|^{\alpha} \|_{L^{\f{3}{1-\alpha}}(\mathbb{R}^3)}  \B\| \f{\Gamma}{ |x_{3}|^{\f\alpha2}r^{\f12}}\B\|_{L^{\f{6}{2+\alpha}}(\mathbb{R}^3)}\B\| \f{\Phi}{ |x_{3}|^{\f\alpha2}r^{\f12}}\B\|_{L^{\f{6}{2+\alpha}}(\mathbb{R}^3)}.
 \ea
 $$
Since $0<\alpha<\f14,$  then $\f12< \f{2+\alpha}{3}$ and $ \f\alpha2<\f{2+\alpha}{6}.$ As a result, we can use  Hardy-Sobolev inequality  \eqref{anHS} and Sobolev embedding theorem agian to get
$$\ba
&\B\| \f{\Gamma}{ |x_{3}|^{\f\alpha2}r^{\f12}}\B\|_{L^{\f{6}{2+\alpha}}(\mathbb{R}^3)}\leq
 C\B\|\Lambda^{\f{\alpha+1}{2}}\Gamma \B\|_{L^{\f{6}{2+\alpha}}(\mathbb{R}^3)}\leq
  C\|\nabla\Gamma  \|_{L^{2}(\mathbb{R}^3)}\\
&\B\| \f{\Phi}{ |x_{3}|^{\f\alpha2}r^{\f12}}\B\|_{L^{\f{6}{2+\alpha}}(\mathbb{R}^3)}\leq C
 \B\|\Lambda^{\f{\alpha+1}{2}}\Phi \B\|_{L^{\f{6}{2+\alpha}}(\mathbb{R}^3)}\leq C
  \|\nabla\Phi  \|_{L^{2}(\mathbb{R}^3)}.
\ea$$
Therefore,
\be\label{k13}
\B|\int_{\mathbb{R}^3} u_{\theta} \frac{\Gamma}{r} \Phi \B|\leq C
 \|u_{\theta} |x_{3}|^{\alpha}\|_{L^{\f{3}{1-\alpha}}(\mathbb{R}^3)}\|\nabla\Phi  \|_{L^{2}(\mathbb{R}^3)} \|\nabla\Gamma\|_{L^{2}(\mathbb{R}^3)}.
 \ee
Plugging \eqref{k11}, \eqref{k12}, \eqref{k13} into
  \eqref{4.1} and  \eqref{4.2}, we find out that
$$\ba &\frac{d}{dt}\big(\|\Phi\|^{2}_{L^{2}(\mathbb{R}^3)}+
\|\Gamma\|^{2}_{L^{2}(\mathbb{R}^3)}\big)+C\big(\|\nabla \Phi\|_{L^{2}(\mathbb{R}^3)}^2+\|\nabla \Gamma\|_{L^{2}(\mathbb{R}^3)}^2\big)\\
  \leq& C
 \|u_{\theta} |x_{3}|^{\alpha}\|_{L^{\f{3}{1-\alpha}}}
 (\|\nabla\Phi\|^2_{L^{2}(\mathbb{R}^3)}+\|\nabla\Gamma\|^2_{L^{2}(\mathbb{R}^3)}).
\ea $$
Because the norm of     $\|u_{\theta} |x_{3}|^{\alpha}\|_{L^{\infty}(0,T;L^{\f{3}{1-\alpha}}(\mathbb{R}^{3}))} $ is sufficiently small, we can obtain
 $\Gamma \in L^{\infty}(0,T;L^{2}(\mathbb{R}^{3}))$.

  At this position, Lemma    \ref{lem2.2}      helps us to complete  the proof of this theorem.
  \end{proof}

  \begin{proof}[Proof of Theorem \ref{the1.5}]
  Taking advantage of the H\"older inequality \eqref{HolderIQ} in mixed-Lorentz space, we have
\be\ba\label{4.16}
\B| \int_{\mathbb{R}^3} u_{\theta} \partial_{r}\frac{u_{r}}{r} \partial_{x_3}\Phi dx\B|\leq C\|r^{\alpha}u_{\theta}\|_{L_{x_{3}}^{p_{3},\infty}
 L_{x_{2}}^{p_{2},\infty}L_{x_{1}}^{p_{1},\infty}(\mathbb{R}^3)}
\B\|\f{\partial_{r}\frac{u_{r}}{r}}{r^{\alpha}}\B\|_{L_{x_{3}}^{\f{2p_{3}}{p_{3}-2},2}
 L_{x_{2}}^{\f{2p_{2}}{p_{2}-2},2}L_{x_{1}}^{\f{2p_{1}}{p_{1}-2},2}(\mathbb{R}^3)}
 \|\partial_{x_3}\Phi\|_{L^{2}(\mathbb{R}^3)}.
\ea\ee
  In view of $p_{1},p_{2}>\f{2}{1-\alpha}$, we can apply the Hardy-Sobolev
  inequality  \eqref{anHS1} in Corollary  \ref{coro} to conclude that
$$
\ba
\B\|\f{\partial_{r}\frac{u_{r}}{r}}{r^{\alpha}}
\B\|_{L_{x_{3}}^{\f{2p_{3}}{p_{3}-2},2}
 L_{x_{2}}^{\f{2p_{2}}{p_{2}-2},2}L_{x_{1}}^{\f{2p_{1}}{p_{1}-2},2}(\mathbb{R}^3)}
 \leq & C \B\|\Lambda^{\alpha}_{x_{1},x_{2}}\partial_{r}\frac{u_{r}}{r}\B\|_{L_{x_{3}}^{\f{2p_{3}}{p_{3}-2},2}
 L_{x_{2}}^{\f{2p_{2}}{p_{2}-2},2}L_{x_{1}}^{\f{2p_{1}}{p_{1}-2},2}(\mathbb{R}^3)}.
\ea$$
 The Sobolev embedding  \eqref{sl} and the Gagliardo-Nirenberg inequality further help   us to derive that
$$
\ba
 \B\|\Lambda^{\alpha}_{x_{1},x_{2}}\partial_{r}\frac{u_{r}}{r}\B\|_{L_{x_{3}}^{\f{2p_{3}}{p_{3}-2},2}
 L_{x_{2}}^{\f{2p_{2}}{p_{2}-2},2}L_{x_{1}}^{\f{2p_{1}}{p_{1}-2},2}(\mathbb{R}^3)}
   \leq& C\B\|\Lambda^{\sum_{i=1}^{3}\f1p_{i}+\alpha} \partial_{r}\frac{u_{r}}{r}\B\|_{L^{2}(\mathbb{R}^3)}\\
 \leq& C\B\| \partial_{r}\frac{u_{r}}{r}\B\|^{1-\alpha-\sum_{i=1}^{3}\f1p_{i}}_{L^{2}(\mathbb{R}^3)}\B\| \nabla\partial_{r}\frac{u_{r}}{r}\B\|_{L^{2}(\mathbb{R}^3)}^{ \alpha+\sum_{i=1}^{3}\f1p_{i}}.
\ea$$
Therefore, we arrive at
$$
\ba
\B\|\f{\partial_{r}\frac{u_{r}}{r}}{r^{\alpha}}\B\|_{L_{x_{3}}^{\f{2p_{3}}{p_{3}-2},2}
 L_{x_{2}}^{\f{2p_{2}}{p_{2}-2},2}L_{x_{1}}^{\f{2p_{1}}{p_{1}-2},2}(\mathbb{R}^3)}
 \leq &
C\| \Gamma\|^{1-\alpha-\sum_{i=1}^{3}\f1p_{i}}_{L^{2}(\mathbb{R}^3)}\| \nabla\Gamma\|_{L^{2}(\mathbb{R}^3)}^{ \alpha+\sum_{i=1}^{3}\f1p_{i}}.
\ea$$
Substituting this   into  \eqref{4.16},  we infer that
 $$\ba
 &\B|\int_{\mathbb{R}^3} u_{\theta} \partial_{r}\frac{u_{r}}{r} \partial_{x_3}\Phi dx\B|\\
\leq&C\|r^{\alpha}u_{\theta}\|_{L_{x_{3}}^{p_{3},\infty}
 L_{x_{2}}^{p_{2},\infty}L_{x_{1}}^{p_{1},\infty}(\mathbb{R}^3)}
\| \Gamma\|^{1-\alpha-\sum_{i=1}^{3}\f1p_{i}}_{L^{2}(\mathbb{R}^3)}\| \nabla\Gamma\|_{L^{2}(\mathbb{R}^3)}^{ \alpha+\sum_{i=1}^{3}\f1p_{i}}
 \|\partial_{x_3}\Phi\|_{L^{2}(\mathbb{R}^3)} \\
\leq& C\|r^{\alpha}u_{\theta}\|_{L_{x_{3}}^{p_{3},\infty}
 L_{x_{2}}^{p_{2},\infty}L_{x_{1}}^{p_{1},\infty}(\mathbb{R}^3)}
 \|\Gamma\|_{L^{2}(\mathbb{R}^3)}^{1-\alpha-\sum_{i=1}^{3}\f1p_{i}}
 (\|\nabla\Phi\|_{L^{2}(\mathbb{R}^3)}
 +\|\nabla\Gamma\|_{L^{2}(\mathbb{R}^3)})^{1+\alpha+\sum_{i=1}^{3}\f1p_{i}}.
 \ea$$
Exactly as in the above derivation, we know that
 $$
  \B|\int_{\mathbb{R}^3} u_{\theta } \partial _{z}\frac {u_{r}}{r} \partial _{r}\Phi \B|
\leq C\|r^{\alpha}u_{\theta}\|_{L_{x_{3}}^{p_{3},\infty}
 L_{x_{2}}^{p_{2},\infty}L_{x_{1}}^{p_{1},\infty}(\mathbb{R}^3)}
\| \Gamma\|^{1-\alpha-\sum_{i=1}^{3}\f1p_{i}}_{L^{2}(\mathbb{R}^3)}\| \nabla\Gamma\|_{L^{2}(\mathbb{R}^3)}^{ \alpha+\sum_{i=1}^{3}\f1p_{i}}
 \|\nabla \Phi\|_{L^{2}(\mathbb{R}^3)} .
$$
Using the H\"older inequality \eqref{HolderIQ} once again, we get
\be\ba\label{4.17}
 &\B| \int_{\mathbb{R}^3} u_{\theta} \frac{\Gamma}{r} \Phi \B|\\=&  \B|\int_{\mathbb{R}^3} u_{\theta} |r|^{\alpha} \f{\Gamma}{ r^{\f{1+\alpha}{2}}}\f{\Phi}{ |r|^{\f{1+\alpha}{2} } } \B|\\
\leq& \|r^{\alpha}u_{\theta}\|_{L_{x_{3}}^{p_{3},\infty}
 L_{x_{2}}^{p_{2},\infty}L_{x_{1}}^{p_{1},\infty}(\mathbb{R}^3)}
\B\|\f{\Gamma}{ r^{\f{1+\alpha}{2}}}\B\|_{L_{x_{3}}^{\f{2p_{3}}{p_{3}-1},2}
 L_{x_{2}}^{\f{2p_{2}}{p_{2}-1},2}L_{x_{1}}^{\f{2p_{1}}{p_{1}-1},2}(\mathbb{R}^3)}
  \B\|\f{\Phi}{ r^{\f{1+\alpha}{2}}}\B\|_{L_{x_{3}}^{\f{2p_{3}}{p_{3}-1},2}
 L_{x_{2}}^{\f{2p_{2}}{p_{2}-1},2}L_{x_{1}}^{\f{2p_{1}}
 {p_{1}-1},2}(\mathbb{R}^3)}.
\ea\ee
According to $\f{1+\alpha}{2}<\f{p_{2}-1}{p_{2}},\f{p_{1}-1}{p_{1}}$, we derive from the Hardy-Sobolev inequality \eqref{anHS1} in Corollary  \ref{coro} and the inclusion relation \eqref{lincreases} that
$$\ba
\B\|\f{\Gamma}{ r^{\f{1+\alpha}{2}}}\B\|_{L_{x_{3}}^{\f{2p_{3}}{p_{3}-1},2}
 L_{x_{2}}^{\f{2p_{2}}{p_{2}-1},2}L_{x_{1}}^{\f{2p_{1}}{p_{1}-1},2}(\mathbb{R}^3)}
 \leq&C
 \B\|\Lambda_{x_{1},x_{2}}^{\f{1+\alpha}{2}}\Gamma \B\|_{L_{x_{3}}^{\f{2p_{3}}{p_{3}-1},2}
 L_{x_{2}}^{\f{2p_{2}}{p_{2}-1},2}L_{x_{1}}^{\f{2p_{1}}{p_{1}-1},2}(\mathbb{R}^3)}.
\ea$$
From the Sobolev embedding  \eqref{sl} and  the Gagliardo-Nirenberg inequality, we get
$$\ba
\B\|\Lambda_{x_{1},x_{2}}^{\f{1+\alpha}{2}}\Gamma \B\|_{L_{x_{3}}^{\f{2p_{3}}{p_{3}-1},2}
 L_{x_{2}}^{\f{2p_{2}}{p_{2}-1},2}L_{x_{1}}^{\f{2p_{1}}{p_{1}-1},2}(\mathbb{R}^3)} \leq& C\B\|\Lambda^{\sum_{i=1}^{3}\f1{2p_{i}}+\f{1+\alpha}{2}}\Gamma \B\|_{L^{2}(\mathbb{R}^3)}
\\ \leq& C\| \Gamma\|_{L^{2}(\mathbb{R}^3)}^{\f{1-\alpha-\sum_{i=1}^{3}\f1p_{i}}{2}}\| \nabla\Gamma\|_{L^{2}(\mathbb{R}^3)}^{\f{1+\alpha+\sum_{i=1}^{3}\f1p_{i}}{2}}.
\ea$$
Hence, there holds
\be\ba\label{4.18}
\B\|\f{\Gamma}{ r^{\f{1+\alpha}{2}}}\B\|_{L_{x_{3}}^{\f{2p_{3}}{p_{3}-1},2}
 L_{x_{2}}^{\f{2p_{2}}{p_{2}-1},2}L_{x_{1}}^{\f{2p_{1}}{p_{1}-1},2}(\mathbb{R}^3)}\leq&
 C\| \Gamma\|_{L^{2}(\mathbb{R}^3)}^{\f{1-\alpha-\sum_{i=1}^{3}\f1p_{i}}{2}}\| \nabla\Gamma\|_{L^{2}(\mathbb{R}^3)}^{\f{1+\alpha+\sum_{i=1}^{3}\f1p_{i}}{2}}.
\ea\ee
Similarly,  we have
\be\ba\label{4.19}
\B\|\f{\Phi}{ r^{\f{1+\alpha}{2}}}\B\|_{L_{x_{3}}^{\f{2p_{3}}{p_{3}-1},2}
 L_{x_{2}}^{\f{2p_{2}}{p_{2}-1},2}L_{x_{1}}^{\f{2p_{1}}{p_{1}-1},2}(\mathbb{R}^3)}
 \leq&C
 \| \Phi\|_{L^{2}(\mathbb{R}^3)}^{\f{1-\alpha-\sum_{i=1}^{3}\f1p_{i}}{2}}\| \nabla\Phi\|_{L^{2}(\mathbb{R}^3)}^{\f{1+\alpha+\sum_{i=1}^{3}\f1p_{i}}{2}}.
\ea\ee
Putting \eqref{4.17},  \eqref{4.18}, \eqref{4.19} together, we obtain
$$\ba
 &\B|\int_{\mathbb{R}^3} u_{\theta} \frac{\Gamma}{r} \Phi\B|
 \\
 \leq& C\|r^{\alpha}u_{\theta}\|_{L_{x_{3}}^{p_{3},\infty}
 L_{x_{2}}^{p_{2},\infty}L_{x_{1}}^{p_{1},\infty}(\mathbb{R}^3)}
  (\| \Gamma\|_{L^{2}}+\| \Phi\|_{L^{2}(\mathbb{R}^3)})^{ 1-\alpha-\sum_{i=1}^{3}\f1p_{i} }\\
  &\times(\| \nabla\Phi\|_{L^{2}}+\| \nabla\Phi\|_{L^{2}(\mathbb{R}^3)})^{ 1+\alpha+\sum_{i=1}^{3}\f1p_{i} }.
 \ea$$
Putting the above estimates together, we arrive at
  $$\ba
 &\frac{d}{dt}\big(\|\Phi\|^{2}_{L^{2}(\mathbb{R}^3)}
 +\|\Gamma\|^{2}_{L^{2}(\mathbb{R}^3)}\big)+C\big(\|\nabla \Phi\|_{L^{2}(\mathbb{R}^3)}^2+\|\nabla \Gamma\|_{L^{2}(\mathbb{R}^3)}^2\big)\\
 \leq&C\|r^{\alpha}u_{\theta}\|_{L_{x_{3}}^{p_{3},\infty}
 L_{x_{2}}^{p_{2},\infty}L_{x_{1}}^{p_{1},\infty}(\mathbb{R}^3)}
  (\| \Gamma\|_{L^{2}(\mathbb{R}^3)}+\| \Phi\|_{L^{2}(\mathbb{R}^3)})^{ 1-\alpha-\sum_{i=1}^{3}\f1p_{i} }\\
  &\times(\| \nabla\Phi\|_{L^{2}(\mathbb{R}^3)}+\| \nabla\Phi\|_{L^{2}(\mathbb{R}^3)})^{ 1+\alpha+\sum_{i=1}^{3}\f1p_{i} }.
 \ea$$

Case 1:
If $ \alpha+\sum_{i=1}^{3}\f1p_{i}<1 $, we conclude by the Young inequality that
 $$\ba
 &\frac{d}{dt}\big(\|\Phi\|^{2}_{L^{2}(\mathbb{R}^3)}+\|\Gamma\|^{2}_{L^{2}(\mathbb{R}^3)}\big)+C\big(\|\nabla \Phi\|_{L^{2}(\mathbb{R}^3)}^2+\|\nabla \Gamma\|_{L^{2}(\mathbb{R}^3)}^2\big)\\
 \leq&
 C\|r^{\alpha}u_{\theta}\|_{L_{x_{3}}^{p_{3},\infty}
 L_{x_{2}}^{p_{2},\infty}L_{x_{1}}^{p_{1},\infty}}
 (\| \Gamma   \|_{L^{2}(\mathbb{R}^3)}^{2}+\| \Phi  \|_{L^{2}(\mathbb{R}^3)}^{2} )+\f18(\| \nabla\Gamma  \|_{L^{2}(\mathbb{R}^3)}^{2}+\|\partial_{x_3}\Phi\|^{2}_{L^{2}(\mathbb{R}^3)}). \ea$$
 The Gronwall inequality and Lemma \ref{lem2.2} entail us to achieve the proof.

Case 2:
 If $ \alpha+\sum_{i=1}^{3}\f1p_{i}=1 $
 $$\ba
& \frac{d}{dt}\big(\|\Phi\|^{2}_{L^{2}(\mathbb{R}^3)}+\|\Gamma\|^{2}_{L^{2}(\mathbb{R}^3)}\big)+C\big(\|\nabla \Phi\|_{L^{2}(\mathbb{R}^3)}^2+\|\nabla \Gamma\|_{L^{2}(\mathbb{R}^3)}^2\big)
 \\\leq& C\|r^{\alpha}u_{\theta}\|_{L_{x_{3}}^{p_{3},\infty}
 L_{x_{2}}^{p_{2},\infty}L_{x_{1}}^{p_{1},\infty}(\mathbb{R}^3)}
  (\| \nabla\Phi\|_{L^{2}}+\| \nabla\Phi\|_{L^{2}(\mathbb{R}^3)})^{ 2 }
 \\\leq& \f12
  (\| \nabla\Phi\|_{L^{2}}+\| \nabla\Phi\|_{L^{2}(\mathbb{R}^3)})^{ 2 },
  \ea$$
  which leads to $\Phi, \Gamma \in L^{\infty}(L^{2}(0,T;\mathbb{R}^{3}))$. At this stage, one can use
  the Lemma \ref{lem2.2} once again to finish the proof of this theorem.
 \end{proof}
\begin{proof}[Proof of Theorem \ref{the1.6}]
According to the result in  \cite{[KP],[NP],[ZZ]}, it suffices  to show that $\f{u_{\theta}}{r}\in L^{4}(0,T;L^{4}(\mathbb{R}^{3}))$. Indeed,
from   $\eqref{ANS}_{1}$ and $\eqref{ANS}_{3}$, we conclude by the standard energy estimate  that
\be\label{4.20}
\f{1}{p}\f{d}{dt}\int_{\mathbb{R}^3} |u_{\theta}|^{\ell}dx +\f{\ell-1}{(\f{\ell}{2})^{2}}
\int_{\mathbb{R}^3} |\nabla u_{\theta}|^{2}dx+\int_{\mathbb{R}^3} \f{|u_{\theta}|^{\ell}}{r}dx= \int_{\mathbb{R}^3} \f1r u_{r}|u_{\theta}|^{\ell}dx.
\ee
The Young inequality ensures that
$$\ba
 \B|\int_{\mathbb{R}^3} \f1r u_{r}|u_{\theta}|^{\ell}dx \B|=&
 \B| \int_{\mathbb{R}^3}  u_{r}r^{\alpha}  |u_{\theta}|^{\ell-\f{\ell(1+\alpha)}{2}}\f{|u_{\theta}|^{\f{\ell(1+\alpha)}{2}}}{r^{1+\alpha}} dx \B|
\\ \leq&  \int_{\mathbb{R}^3} |u_{r}r^{\alpha}|^{\f{2}{1-\alpha}} |u_{\theta}|^{\ell}+\f18\int_{\mathbb{R}^3} \f{|u_{\theta}|^{\ell} }{r^{2}}dx.
\ea$$
Taking advantage of the H\"older inequality, we see that
$$\ba
 \B|\int_{\mathbb{R}^3} (u_{r}r^{\alpha})^{\f{2}{1-\alpha}} |u_{\theta}|^{\ell}dx \B|
\leq&   C\||u_{r}r^{\alpha}|^{\f{2}{1-\alpha}}
\|_{L^{\f{\overrightarrow{p}}{\f{2}{1-\alpha}},\infty}(\mathbb{R}^3)}  \||u_{\theta}|^{\f{\ell}{2}}\|^{2}_{L^{\f{2p_{1}(1-\alpha)}{p_{1}(1-\alpha)-2},2}
L^{\f{2p_{2}(1-\alpha)}{p_{2}(1-\alpha)-2},2}L^{\f{2p_{3}(1-\alpha)}{p_{3}(1-\alpha)-2},2}(\mathbb{R}^3)}.
\ea
$$
It follows from the Sobolev embedding theorem \eqref{sl} and the interpolation inequality that
$$\ba
 \||u_{\theta}|^{\f{\ell}{2}}\|_{L^{\f{2p_{1}(1-\alpha)}{p_{1}(1-\alpha)-2},2}
L^{\f{2p_{2}(1-\alpha)}{p_{2}(1-\alpha)-2},2}L^{\f{2p_{3}(1-\alpha)}{p_{3}(1-\alpha)-2},2}(\mathbb{R}^3)}
\leq& C\|\Lambda^{\sum_{i}\f{2}{p_{i}(1-\alpha)}}|u_{\theta}|^{\f{\ell}{2}}\|_{L^{2}(\mathbb{R}^3)}\\
\leq&
C\||u_{\theta}|^{\f{\ell}{2}}\|^{ 1-\sum_{i}\f{1}{p_{i}(1-\alpha)}}_{L^{2}(\mathbb{R}^3)}
\|\nabla|u_{\theta}|^{\f{\ell}{2}}\|_{L^{2}(\mathbb{R}^3)}^{\sum_{i}\f{1}{p_{i}(1-\alpha)}}.
\ea$$
Based on this, there holds that
$$
\int_{\mathbb{R}^3} (u_{r}r^{\alpha})^{\f{2}{1-\alpha}} |u_{\theta}|^{\ell}\leq C\| u_{r}r^{\alpha}
\|^{\f{2}{1-\alpha}}_{L^{ \overrightarrow{p} ,\infty}} \||u_{\theta}|^{\f{\ell}{2}}\|^{ 2-\sum_{i}\f{2}{p_{i}(1-\alpha)}}_{L^{2}(\mathbb{R}^3)}
\|\nabla|u_{\theta}|^{\f{\ell}{2}}
\|_{L^{2}(\mathbb{R}^3)}^{\sum_{i}\f{2}{p_{i}(1-\alpha)}}.
$$
Substituting this into \eqref{4.20}, we know that
$$\ba
 &\f{1}{p}\f{d}{dt}\int_{\mathbb{R}^3} |u_{\theta}|^{\ell}dx +\f{\ell-1}{(\f{\ell}{2})^{2}}
\int_{\mathbb{R}^3} |\nabla u_{\theta}|^{2}dx+\int_{\mathbb{R}^3} \f{|u_{\theta}|^{\ell}}{r}dx\\
\leq& C
\| u_{r}r^{\alpha}
\|^{\f{2}{1-\alpha}}_{L^{ \overrightarrow{p},\infty }}   \||u_{\theta}|^{\f{\ell}{2}}\|^{ 2-\sum_{i}\f{2}{p_{i}(1-\alpha)}}_{L^{2}}
\|\nabla|u_{\theta}|^{\f{\ell}{2}}\|_{L^{2}}
^{\sum_{i}\f{2}{p_{i}(1-\alpha)}}+\f18\int_{\mathbb{R}^3} \f{|u_{\theta}|^{\ell} }{r^{2}}dx.
\ea$$
With this in hand,
by arguing as was done to prove   previous theorems, we complete the proof of this theorem.
\end{proof}
\section*{Acknowledgements}

Wang was partially supported by  the National Natural
 Science Foundation of China under grant (No. 11971446, No. 12071113   and  No.  11601492). Wei was partially supported by the National Natural Science Foundation of China under grant (No. 11601423, No. 11871057).
 Yu was partially supported by the
National Natural Science Foundation of China (NNSFC) (No. 11901040), Beijing Natural
Science Foundation (BNSF) (No. 1204030) and Beijing Municipal Education Commission
(KM202011232020).


\begin{thebibliography}{00}
\bibitem{[BT]} M. Badiale and G. Tarantello, A Sobolev-Hardy inequality with application to  a nonlinear elliptic equation arising in astrophysics, Arch. Ration. Mech. Anal., 163 (2002), 259-293.
    \bibitem{[BCD]}H. Bahouri, J. Y. Chemin and R. Danchin, Fourier Analysis and Nonlinear Partial Differential Equations, Grundlehren dermathematischen Wissenschaften 343, Springer-Verlag, 2011.

\bibitem{[BP]}
A. Benedek and R. Panzone,  The space $L^{p}$ with mixed norm.  Duke Math.
J.   28   (1961)  301--324.
\bibitem{[Blozinski]}
A. P. Blozinski, Multivariate rearrangements and Banach function spaces with mixed norms. Trans. Amer. Math. Soc. 263 (1981),  149--167.
\bibitem{[BV]}
H. Br\'ezis and  J. L. Vazquez, Blow-up solutions of some nonlinear elliptic problems, Rev. Mat. Complut. 10  (1997) 443--469.

\bibitem{[CKN]}L. Caffarelli, R. Kohn, L. Nirenberg, Partial regularity of suitable weak solution of the Navier-Stokes equations, Comm. Pure Appl. Math., 35, 771--837, 1982.

\bibitem{[CL]} D. Chae and  J. Lee, On the regularity of axisymmetric solutions of the Navier-Stokes equations, Math. Z., 239 (2002), 645-671.
\bibitem{CSTY1} C. C. Chen, R. M. Strain, H. T. Yau and T. P. Tsai,
Lower bound on the blow-up rate of the axisymmetric Navier-Stokes
equations, Int. Math Res. Notices 8, 31 (2008).


\bibitem{CSTY2}  C. C. Chen, R. M. Strain,  T. P. Tsai and H. T. Yau,
Lower bound on the blow-up rate of the axisymmetric Navier-Stokes
equations II, Comm. Partial Diff. Eqns., 34 (2009), 203--232.

\bibitem{[CFZ]} H. Chen, D. Y. Fang and T. Zhang, Regularity of 3D axisymmetric Navier-Stokes equations, Discrete Contin. Dyn. Syst., 37 (2017), 1923--1939.
    \bibitem{[CFZ1]}H. Chen, D. Y. Fang and  T. Zhang,
 Muliti-scale regularity of axisymmetric Navier-Stokes
equations.  arXiv:1802.08956v1. 2018.
    \bibitem{[CTZ]} H.  Chen, T.  Tsai  and  T. Zhang, Remarks on local regularity of axisymmetric solutions to the 3D Navier-Stokes equations, Comm. Partial Diff. Eqns.,   (2022).  see also arXiv:2201.01766.
\bibitem{[Fernandez]}
D. L. Fernandez, Lorentz spaces, with mixed norms, J. Funct. Anal. 25 (1977), 128--146.

\bibitem{[FS]}
R. Frank and R. Seringer, Nonlinear ground state representation and sharp Hardy inequalities, J. Funct. Anal. 255 (2008), 3407--3430.

\bibitem{[GZ]}
J.A. Goldstein and Q. S.  Zhang, On a degenerate heat equation with a singular potential, J. Funct. Anal. 186 (2001) 342--359.

\bibitem{[Hardy]}
G. Hardy, J.  E.  Littlewood  and G. Polya,  Inequalities, Cambridge Univ. Press,
Cambridge, UK, 1934.
\bibitem{[HYZ]}H. Hajaiej, X. Yu and Z. Zhai, Fractional Gagliardo-Nirenberg and Hardy inequalities under Lorentz norms, J. Math. Anal. Appl., 396 (2012), 569--577.
     \bibitem{[KT]}
D. Q. Khai and N. M. Tri, Solutions in mixed-norm Sobolev-Lorentz spaces to the initial value problem for the Navier-Stokes equations,  J.   Math. Anal.   Appl.,  417 (2014),  819--833.
\bibitem{KNSS}G. Koch, N. Nadirashvili, G. A. Seregin, and V.  \v{S}ver\'ak, Liouville theorems for
the Navier-Stokes equations and applications,  Acta Math.,   203  (2009),  83--105.

\bibitem{[KP]}
O. Kreml and M. Pokorn\'y,  A regularity criterion for the angular velocity component in axisymmetric
Navier-Stokes equations, Electron. J. Differential Equations.    08  (2007),    10 pp.
\bibitem{[KPZ]}
A. Kubica , M. Pokorn\'y and W. Zajaczkowski, Remarks on regularity criteria for axially symmetric
weak solutions to the Navier-Stokes equations, Math. Methods Appl. Sci. 35 (2012), 360--371.













\bibitem{[L]} O. A. Ladyzhenskaya, Unique global solvability of the
three-dimensional Cauchy problem for the Navier-Stokes equations in
the presence of axial symmetry, Zap. Naucn. Sem. Leninggrad. Otdel.
Math. Inst. Steklov., 7 (1968), 155-177, (Russian).
\bibitem{[LZ]} Z. Lei and Q. S. Zhang, A Liouville theorem for the
axially-symmetric Navier-Stokes equations, J. Funct. Anal., 261 (2011), 2323--2345.

\bibitem{[LZ2]} Z. Lei and Q. S. Zhang, Criticality of the axially symmetric Navier-Stokes equations, 289 (2017),   169--187.


\bibitem{[LMNP]}S. Leonardi, J. M\'{a}lek, J. Ne\v{c}as and M. Pokorn\'{y},  On axially symmetric flows in $\mathbb{R}^3,$ Z. Anal. Anwendungen 18 (1999), 639--649.


\bibitem{[MZ]} C. Miao and  X. Zheng, On the global well-posedness for the Boussinesq system with horizontal dissipation, Comm. Math. Phys., 321 (2013), 33--67.

     \bibitem{[MWZ]}
  C. Miao, J. Wu and Z. Zhang, Littlewood-Paley theory and applications to fluid dynamics equations, Monographs
on Modern pure mathematics, No.142. Science Press, Beijing, 2012.


\bibitem{[NP]}
J. Neustupa and M.  Pokorn\'y,   Axisymmetric flow of Navier-Stokes fluid in the whole space with non-zero. Math.
Bohem., 126 (2001), 469--481.
\bibitem{[Neil]}
R. O'Neil, Convolution operaters and $L^{p,q}$ spaces. Duke Math J., 30 (1963),  129--142.
\bibitem{[P]}X. Pan,  Regularity of solutions to axisymmetric Navier-Stokes equations with a slightly supercritical condition, J. Differential Equations, 260 (2016), 8485--8529.

 \bibitem{[Tao]}    T. Tao, Nonlinear Dispersive Equations, Local and Global Analysis, CBMS Reg. Conf. Ser. Math.,
vol. 106, Amer. Math. Soc., Providence, RI,  , 2006.
\bibitem{[UY]} M. R. Uchovskii and B. I. Yudovich, Axially symmetric flows of an
ideal and viscous fluid, J. Appl. Math. Mech., 32 (1968), 52--61.


\bibitem{[WXW]} Y.  Wang, X. Mei  and W. Wei. Gagliardo-Nirenberg inequality in anisotropic Lebesgue spaces and energy equality in the Navier-Stokes equations.  arXiv:2204.07479 (2022).
\bibitem{[Wei]} D. Y. Wei, Regularity criterion to the axially symmetric Navier-Stokes equations, J. Math. Anal. Appl., 435 (2016), 402--413.
\bibitem{[WWY]}   W. Wei, Y. Wang and Y. Ye, Gagliardo-Nirenberg inequalities in Lorentz type spaces and energy equality for the Navier-Stokes system.  arXiv: 2106.11212  (2021).








     \bibitem{[Yu]}
H. Yu, An anisotropic Sobolev-Hardy inequality with
application to 3D axisymmetric Navier-Stokes equations, Appl. Anal. 99 (2020),  313--325.


\bibitem{[ZZ]}P. Zhang and T. Zhang, Global axisymmetric solutions to three-dimensional Navier-Stokes system, Int. Math. Res. Not., 2014 (2014), 610--642.
     \bibitem{[zz]}  J.  Zhang  and J. Zheng, Scattering theory for nonlinear Schr\"odinger equations with inverse-square potential.  J. Funct. Anal. , 267 (2014),  2907--2932.

 \end{thebibliography}
\end{document}